\documentclass[12pt, a4paper]{article}

\usepackage{amsfonts}
\usepackage{amsmath}
\usepackage{amssymb}
\usepackage[english]{babel}
\usepackage{amsthm}
\usepackage{fourier}
\usepackage{graphicx}
\usepackage{color}

\theoremstyle{plain}

\newtheorem {thm}{Theorem}[section]
\newtheorem {lem}[thm]{Lemma}

\theoremstyle{definition}

\title{The sign-sequence constant of the plane  \footnotetext{Work on this paper by both authors was supported by ERC Grant 267165 DISCONV. Thanks to Imre B\'ar\'any for suggesting the problem, and to Jesus Jeronimo and Steven Karp for comments on an earlier draft.}}

\author{
Ben Lund \thanks{
Rutgers University, USA.
{\sl lund.ben@gmail.com}}
\and
Alexander Magazinov \thanks{%
Tel Aviv University, %
Israel.
{\sl magazinov@post.tau.ac.il}
}}

\begin{document}

\maketitle

\begin{abstract}
Let $L$ be a finite-dimensional real normed space, and let $B$ be the unit ball in $L$.
The sign sequence constant of $L$ is the least $t>0$ such that, for each sequence $v_1, \ldots, v_n \in B$, there are signs $\varepsilon_1, \ldots, \varepsilon_n \in \{-1, +1\}$ such that $\varepsilon_1 v_1 + \ldots + \varepsilon_k v_k \in t B$, for each $1 \leq k \leq n$.

We show that the sign sequence constant of a plane is at most $2$, and the sign sequence constant of the plane with the Euclidean norm is equal to $\sqrt{3}$.
\end{abstract}
\section{Introduction}

Throughout this note, $L$ denotes a finite-dimensional real normed space, and $B$ denotes the unit ball in $L$.

Suppose we are given a sequence of vectors $v_1, v_2, \ldots, v_n \in B$, and we want to find signs $\varepsilon_1, \varepsilon_2, \ldots, \varepsilon_n \in \{-1, +1\}$, such that the norm of each signed partial sum $\varepsilon_1 v_1 + \ldots + \varepsilon_k v_k$ is bounded by a constant depending only on $L$.
We term the smallest constant for which this is always possible the sign-sequence constant, and denote it by $SS(L)$.

The question of determining the sign-sequence constant is closely related to the analogous question of determining the best possible constant in Steinitz's Lemma.
Suppose we are given a finite set $V \subset B$ with $\sum_{v \in V} = 0$, and we want to find an ordering $v_1, \ldots, v_n$ of the elements of $V$ such that the norm of each partial sum $v_1 + \ldots + v_k$ is bounded by a constant depending only on $L$.
We term the smallest constant for which is always possible the Steinitz constant, and denote it by $S(L)$.
The fact that $S(L)$ exists for an arbitrary norm was first proved by Steinitz \cite{steinitz1913bedingt}, answering a question of Riemann and L\'evy, and is usually called Steinitz's Lemma.

In the general case, the results on the sign-sequence constant and the methods used to prove them are are similar to those for the Steinitz constant.
Grinberg and Sevast'yanov \cite{grinberg1980value} showed that $S(L) \leq d$, and B\'ar\'any and Grinberg \cite{barany1981some} used a similar argument to show that $SS(L) \leq 2d-1$.
Chobanyan \cite{chobanyan1994convergence} showed that $S(L) \leq SS(L)$, so an upper bound on $SS(L)$ immediately implies the same bound for $S(L)$.
The linear upper bound on $S(L)$ is tight for the $l_1$ metric (up to a constant factor), as shown by the set consisting of the vectors $e_1, \ldots, e_d$ along with the vector $-(e_1 + \ldots + e_d)/d$ repeated $d$ times.
It is a notorious open problem to prove a sub-linear upper bound on $S(L)$ or $SS(L)$ for the Euclidean or max norms.
A good survey on these problems and results is \cite{barany2008power}.

It is known that the Steinitz constant for any real normed plane is bounded above by $3/2$, and the Steinitz constant for the Euclidean plane is equal to $\sqrt{5}/2$  -- see Bergstr\"om \cite{bergstrom1931zwei} and Banaszczyk \cite{banaszczyk1987steinitz}.
Swanepoel \cite{swanepoel2000balancing} showed that, for a sequence of unit vectors $v_1, v_2, \ldots, v_n$  in a real normed plane, there exist signs $\varepsilon_1, \varepsilon_2, \ldots, \varepsilon_n$,  such that the norm of $\varepsilon_1 v_1 + \ldots + \varepsilon_k v_k$ for even $k$ is at most $2$, and at most $\sqrt{2}$ for the Euclidean norm.
In this note, we prove the following bounds on the sign-sequence constant.

\begin{thm}\label{thm:main-arbitrary}
Let $L$ be a two-dimensional real normed space.
Then
$$SS(L) \leq 2.$$
\end{thm}

\begin{thm}\label{thm:main-euclidean}
Let $\mathbb{E}^2$ be the real plane with Euclidean norm. Then,
$$SS(\mathbb{E}^2) = \sqrt{3}.$$
\end{thm}

The upper bound of $SS(L) \leq 2$ is easily seen to be tight for the max norm and for the $l_1$ norm.
For instance, the vector sequence $(1,0), (0,1)$ shows that Theorem \ref{thm:main-arbitrary} is tight for the $l_1$ norm; since $SS(L)$ is affinely invariant, and the $l_1$ and max norms are affinely eqivalent in the plane, tightness for the max norm follows.
The construction that shows that $SS(\mathbb{E}^2) \geq \sqrt{3}$ is new, and presented in section \ref{sec:construction}.

It is an open question whether a near-converse of Chobanyan's theorem holds; in particular, does there exist a constant $C$ such that $SS(L) \leq C \cdot S(L)$?
Since $S(\mathbb{E}^2) = \sqrt{5}/2$, Theorem \ref{thm:main-euclidean} shows that this converse cannot be true with $C = 1$.

In the proofs of Theorems \ref{thm:main-arbitrary} and \ref{thm:main-euclidean}, we introduce two new constructions that may have further interest: \textit{trapping families} and \textit{admissible sets}.
While it might be possible to use trapping families to make progress on the problem of showing that $SS(\mathbb{E}^n)$ has a sublinear dependence on $n$, we show in section \ref{sec:higherDims} that trapping families derived from admissible sets (which give the tight bound in the plane) cannot be used to to prove sublinear dependence on dimension.

\section{Trapping families}

Let $\mathcal F$ be a family of finite non-empty centrally symmetric sets in $L$. Assume that the following
two conditions hold.
\begin{enumerate}
	\item $\{ \mathbf 0 \} \in \mathcal F$.
	\item For every $F \in \mathcal F$ and every vector $v \in L$ with $\| v \| \leq 1$, there exists a set $F' \in \mathcal F$ such that
				$$(F + v) \cup (F - v) \supseteq F'.$$
\end{enumerate}
Then the family $\mathcal F$ will be called {\it trapping}.

\begin{lem}
Let $\mathcal F$ be a trapping family. Let
$$r(\mathcal F) = \sup\limits_{F \in \mathcal F} \sup\limits_{v \in F} \| v \|.$$
Then $SS(L) \leq r(\mathcal F)$.
\end{lem}

\begin{proof}
Let $v_1, v_2, \ldots, v_n$ be a sequence of vectors such that $\| v_i \| \leq 1$. Construct a sequence
$$F_0, F_1, \ldots, F_n \qquad (F_i \in \mathcal F)$$
inductively as follows. Set $F_0 = \{ \mathbf 0 \}$. If $F_i$ is constructed, choose $F_{i + 1}$ so that
$$(F_i + v_{i + 1}) \cup (F_i - v_{i + 1}) \supseteq F_{i + 1}.$$
This is always possible, since $\mathcal F$ is a trapping family.

Now we prove the following claim. For every $i$ ($0 \leq i \leq n$) and every $f \in F_i$ there is an identity
$$f = \varepsilon_1 v_1 +\varepsilon_2 v_2 + \ldots + \varepsilon_i v_i,$$
with $\varepsilon_1, \varepsilon_2, \ldots, \varepsilon_i \in \{-1, +1\}$,
such that every $j$-th ($0 \leq j < i$)  partial sum of the right-hand part belongs to $F_j$.

To prove the claim we proceed by induction over $i$. The case $i = 0$ is clear.

Let $0 < i \leq n$. For every $f \in F_i$ we have $f = f' + \varepsilon_i v_i$, where $f' \in F_{i - 1}$.
By the inductive assumption for $i - 1$, we have
$$f' = \varepsilon_1 v_1 + \varepsilon_2 v_2 + \ldots + \varepsilon_{i-1} v_{i - 1},$$
so that every $j$-th ($0 \leq j < i - 1$) partial sum of the right-hand part belongs to $F_j$. Adding $\varepsilon_i v_i$,
we get
$$f = \varepsilon_1 v_1 + \varepsilon_2 v_2 + \ldots + \varepsilon_i v_i.$$
Every $j$-th ($0 \leq j < i - 1$) partial sum of the right-hand part is unchanged and therefore belongs to $F_j$.
The $(i - 1)$-th partial sum is $f' \in F_{i - 1}$. Hence the claim is proved for $i$.

The statement of lemma immediately follows from the claim for $i = n$.

\end{proof}

\section{Trapping families of a special type}

We call a set of vectors $v_1, v_2, \ldots, v_k \in L$ {\it admissible}, if $\|v_i\| \leq 1$, and
$$\| \varepsilon_1 v_1 + \varepsilon_2 v_2 + \ldots + \varepsilon_k v_k \| > 1$$
for every sequence of coefficients $\varepsilon_i \in \{ -1, 0, 1 \}$ with at least two non-zero terms. An empty set is always admissible.

If $v_1, v_2, \ldots, v_k \in L$, we will write $\Pi(\{ v_1, v_2, \ldots, v_k \})$ for the set of all linear combinations
$$\pm v_1 \pm v_2 \pm \ldots \pm v_k$$
(for all possible sign patterns). By definition, set $\Pi(\varnothing) = \{ \mathbf 0 \}$.

\begin{lem}
Let $\mathcal A(L)$ be the family of all admissible sets in $L$. Define
$$\mathcal F(L) = \left\{ \Pi(\{ v_1, v_2, \ldots, v_k \}) : \{ v_1, v_2, \ldots, v_k \} \in \mathcal A(L) \right\}.$$
Then $\mathcal F(L)$ is a trapping family.
\end{lem}

\begin{proof}
Let $F \in \mathcal F(L)$. Then $F = \Pi(\{ v_1, v_2, \ldots, v_k \})$, where $\{ v_1, v_2, \ldots, v_k \}$ is an admissible set of vectors.

Let $\| v \| \leq 1$. Then there are two cases.

\noindent {\bf Case 1.} The set $\{ v_1, v_2, \ldots, v_k, v \}$ is admissible. Then
$$(F + v) \cup (F - v) = \Pi(\{ v_1, v_2, \ldots, v_k, v \}) \in \mathcal F(L).$$

\noindent {\bf Case 2.} The set $\{ v_1, v_2, \ldots, v_k, v \}$ is not admissible. Then
$$\| v + \varepsilon_1 v_1 + \varepsilon_2 v_2 + \ldots + \varepsilon_k v_k \| \leq 1$$
for some choice of $\varepsilon_i \in \{ -1, 0, 1 \}$. Consider such a combination with the greatest possible number of non-zero terms.

Without loss of generality, we can assume that the expression is
$$\| v + v_1 + v_2 + \ldots + v_m \| \leq 1.$$
Indeed, we may permute $v_i$'s as well as replace any $v_i$ with $-v_i$ and $v$ with $-v$.

Then the set $\{ v + v_1 + v_2 + \ldots + v_m, v_{m + 1}, v_{m + 2}, \ldots, v_k \}$ is admissible and
$$(F + v) \cup (F - v) \supseteq \Pi(\{ v + v_1 + v_2 + \ldots + v_m, v_{m + 1}, v_{m + 2}, \ldots, v_k \}),$$
where
$$\Pi(\{ v + v_1 + v_2 + \ldots + v_m, v_{m + 1}, v_{m + 2}, \ldots, v_k \}) \in \mathcal F(L).$$

\end{proof}

\section{Two-dimensional admissible sets}

We show that admissibility in the two-dimensional space is a restrictive condition for set of vectors. For instance, the
cardinality of such a set is at most 2.
Once this is established, the upper bounds of Theorems ~\ref{thm:main-arbitrary} and \ref{thm:main-euclidean} follow easily.

\begin{lem}\label{lem:2d-admissible}
Let $L$ be a two-dimensional normed linear space. If $V$ is an admissible set of vectors in $L$, then $|V| \leq 2$.
\end{lem}

\begin{proof}
Since a subset of an admissible set is admissible, it is enough to prove that there is no 3-element admissible set.

Assume that there exists a 3-element admissible set $\{u, v, w\}$. There is a non-trivial linear dependence between $u$, $v$ and $w$:
$$au + bv + cw = \mathbf 0.$$
Since we can permute $u$, $w$ and $w$, and also change signs of any of them without affecting admissibility, we can assume that
$$1 = a \geq b \geq c \geq 0.$$

We claim that $\| u + v \| \leq 1$. Indeed,
$$u + v = u + v - (au + bv + cw) = (1 - b)v - cw.$$
But $1 - b \geq 0$, $c \geq 0$. Hence
$$\| (1 - b)v - cw \| \leq |1-b| \, \|v\| + |c| \, \|w\| \leq 1 - b + c \leq 1.$$
This is a contradiction, because admissibility requires $\| u + v \| > 1$.

\end{proof}

Theorem~\ref{thm:main-arbitrary} follows immediately, since each member of $\mathcal F(L)$ is the sum of the vectors in an admissible set.
In order to prove Theorem ~\ref{thm:main-euclidean}, we will show that the Euclidean norm of the sum of two vectors in an admissible set is less than $\sqrt{3}$.

Indeed, let $\{v_1, v_2\} \in \mathcal A(\mathbb{E}^2)$.
By the parallelogram equality,
$$
\| v_1 + v_2 \|^2 = 2 \|v_1\|^2 + 2 \|v_2 \|^2 - \|v_1 - v_2\|^2 < 3.
$$

\section{Lower bound on $SS(\mathbb{E}^2)$}\label{sec:construction}

We describe a family of sets of vectors in $\mathbb{E}^2$ that shows $SS(\mathbb E^2) \geq \sqrt{3}$.

\begin{lem}
For any $\delta > 0$, there is a natural number $n = n(\delta)$ and a sequence of vectors $v_1, \ldots, v_n \in \mathbb{E}^2$ with $\| v_i \| \leq 1$, such that for any sequence of signs $\varepsilon_1, \ldots, \varepsilon_n \in \{-1,+1\}$, there is a partial sum $\|\varepsilon_1 v_1 + \ldots + \varepsilon_k v_k\| \geq \sqrt{3} - \delta$.
\end{lem}
\begin{proof}
If $\sqrt{3} - \delta \leq \sqrt{2}$, the sequence $v_1 = (1,0), v_2 = (0,1)$ works, so we suppose $\delta < \sqrt{3} - \sqrt{2}$.

Given $x$ with $1 \leq \| x \| < \sqrt{2}$, we claim that there is a vector $v$ with $\| v \| = 1$ such that $\|x + v \| \geq \sqrt{3} - \delta$ and $\| x - v \| > \|x\| + \delta$.
Choose $v$ so that $\|x + v \| = \sqrt{3} - \delta$.
By the parallelogram law,
\begin{align*}
 \| x - v \|^2 &= 2\|x\|^2 + 2\|v\|^2 - \|x + v\| ^2, \\
&= 2\|x \|^2 - 1 + 2 \delta \sqrt{3} - \delta^2.
\end{align*}

Using the assumptions that $\|x \|^2 \geq 1$ and $\delta^2 < (\sqrt{3} - \sqrt{2})\delta$, we have
\[\|x - v \|^2 > \|x\|^2 + 2 \sqrt{2} \delta + \delta^2.\]
Since $\|x \| < \sqrt{2}$, this implies that $\|x-v\| > \|x\| + \delta$.

We build $v_1, \ldots, v_n$ inductively as follows.
Let $v_1 = (0,1)$.
Given $v_1, \ldots, v_k$, let $x_k = -(v_1 + \ldots + v_k)$, and choose $v_{k+1}$ by the above procedure so that $\|x_k + v_{k+1}\| \geq \sqrt{3} - \delta$, and $\|x_k - v_{k+1} \| \geq \|x_k\| + \delta$.
If $\|x_{k} - v_{k+1}\| \geq \sqrt{2}$, let $n=k+2$ and choose $v_{n}$ of unit length perpendicular to $x_{k+1}$.
Clearly, $n < 3 + (\sqrt{2} - 1)/\delta$.

Let $\varepsilon_1, \ldots, \varepsilon_n \in \{-1, +1\}$ be a sequence of signs so that $\varepsilon_1 = -1$; we may assume $\varepsilon_1 = -1$ without loss of generality, since flipping all of the signs does not affect the norm of any partial sum.
If there is an $i$ such that $\varepsilon_i = 1$, let $k$ be the first index such that $\varepsilon_k = 1$.
Then $\| \sum_{i \leq k} \varepsilon_i v_i \| \geq \sqrt{3} - \delta$.
If $\varepsilon_i = -1$ for all $i$, then $\| \sum_{i \leq n} \varepsilon_i v_i \| \geq \sqrt{3} > \sqrt{3} - \delta$.

\end{proof}

\section{Admissible sets in higher dimensions}\label{sec:higherDims}

We show that admissible sets cannot be used in the framework of trapping families to show that the sign sequence constant for the Euclidean or max norms has a sub-linear dependence on the dimension.

For $L$ taken to be $\mathbb{R}^d$ with either the max or Euclidean norm, we give an explicit set of vectors $V \in \mathcal A (L)$ such that $\| \sum_{v \in V} v \| \geq c d$ in the relevant norm, for a constant $c>0$ depending on the norm.
Since $\sum_{v \in V}v \in \Pi(V) \in \mathcal F(L)$, this shows that $r(\mathcal F(L)) \geq c d$.

For a vector $v\in \mathbb{R}^d$, denote by $v[i]$ the $i$th coordinate of $v$.

For $1 \leq i \leq d-1$, let $v_i \in \mathbb{R}^d$ be the vector such that $v_i[i] = -1$ and $v_i[j] = 1$ for $j \neq i$.
Then $v_1, \ldots, v_{d-1}$ is an admissible set for $\mathbb{R}^d$ with the max norm, and $\| \sum v_i \| = d-1$.

For $1 \leq i \leq d-1$, let $v_i \in \mathbb{R}^d$ be the vector such that $v_i[1] = 0.2$, $v_i[i + 1] = 0.8$, and $v_i[j] = 0$ for
$j \notin \{1, i + 1\}$.
Then $v_1, \ldots, v_{d-1}$ is an admissible set for $\mathbb{R}^d$ with the Euclidean norm, and $\| \sum v_i \| > 0.2(d - 1)$.

\end{document}